\documentclass{amsart}

\usepackage[T1]{fontenc}
\usepackage{mathpazo}
\usepackage{microtype}
\usepackage{fixltx2e}

\usepackage{amsmath, amsthm, amssymb}
\usepackage[all]{xy}

\usepackage[backend=bibtex]{biblatex}
\bibliography{\jobname.bib}

\newtheorem{theorem}{Theorem}
\newtheorem{proposition}[theorem]{Proposition}
\theoremstyle{definition}
\newtheorem{definition}[theorem]{Definition}

\renewcommand{\L}{\mathcal{L}}
\newcommand{\h}{\mathfrak{h}}
\renewcommand{\l}{\mathfrak{l}}
\renewcommand{\r}{\mathfrak{r}}
\renewcommand{\a}{\mathfrak{a}}
\newcommand{\q}{\mathfrak{q}}

\DeclareMathOperator{\Span}{Span}
\DeclareMathOperator{\Ker}{Ker}

\begin{document}

\title{The matrix {L}ie algebra on a one-step ladder is zero product
  determined}

\author{Daniel Brice}
\address{
  Department of Mathematics\\
  California State University, Bakersfield\\
  Bakersfield, California 93311, USA
}
\email{daniel.brice@csub.edu}

\date{16 October 2015 (Revised 2 December 2015)}

\begin{abstract}
  The class of \emph{matrix algebras on a ladder $\L$} generalizes the
  class of block upper triangular matrix algebras.
  It was previously shown that the matrix algebra on a ladder $\L$ is
  zero product determined under matrix multiplication.
  In this article, we show that the matrix algebra on a one-step
  ladder is zero product determined under the Lie bracket.
\end{abstract}

\maketitle

\section{Introduction}

In \cite{brice2015zero}, the authors defined a class of matrix algebras,
the \emph{ladder matrix algebras}, that generalizes the class of block
upper triangular matrix algebras. They introduce the notion of an upper
triangular $k$-step ladder as a method of parameterizing and indexing
these algebras. Certain one-step ladder matrix algebras arise as ideals
of derivation algebras of parabolic subalgebras of reductive Lie
algebras, which provided the motivation for their study
\cite{brice2014derivation}.

While these terms are made precise in Section \ref{sec:prelim},
the concepts are perhaps best illustrated with an example.
Let $\L = \{ (3,2), (6,5) \}$.
$\L$ is then a $2$-step upper triangular ladder on $6$.
The ladder matrix algebra on $\L$ is the subalgebra
\[
  M_\L = \left\{
  \begin{pmatrix}
    0 & * & * & * & * & * \\
    0 & * & * & * & * & * \\
    0 & * & * & * & * & * \\
    0 & 0 & 0 & 0 & * & * \\
    0 & 0 & 0 & 0 & * & * \\
    0 & 0 & 0 & 0 & * & *
  \end{pmatrix}
  \right\}
\]
of $M^{n \times n}$.

An algebra $(A, \ast)$ is \emph{zero product determined} if each
bilinear map $\varphi$ on $A \times A$ that preserves zero products
necessarily factors as a linear map $f$ on $A^2$ composed with the
algebra multiplication $\ast$ so that $\varphi(x,y) = f(x \ast y)$. The
notion is motivated by the linear preserver problem in operator theory
and has recently become a topic of considerable research
\cite{brevsar2009zero}. It was previously shown that the ladder matrix
algebras are zero product determined when $\ast$ is matrix
multiplication \cite{brice2015zero}. The purpose of this paper is to
show that a one-step ladder matrix algebra is zero product determined
when $\ast$ is the Lie bracket $[x,y] = xy - yx$.

Previous work on zero product determined algebras has also considered
the case where $\ast$ is the Jordan product $x \circ y = xy + yx$
\cite{brevsar2009zero}. Extending the present results on ladder matrix
algebras to this case, and to the $k$-step case for both the Lie bracket
and the Jordan product, remains a topic of interest to the author.

\section{Preliminaries}
\label{sec:prelim}

Let $F$ be a field.
Let $n$ be a positive integer.
Let $M^{n \times n}_F$ denote the space of $n$-by-$n$ matrices with
entries in $F$.
Let $e_{i,j}$ denote the matrix whose entry in the $i$th row $j$th
column is $1_F$, and whose other entries are $0_F$.

\begin{definition}
  A $k$-step \emph{ladder} on $n$ is a set of pairs of positive integers
  \[
    \L = \{ (i_1, j_1), ..., (i_k, j_k) \}
  \]
  with $1 \leq i_1 < i_2 < ... < i_k \leq n$
  and $1 \leq j_1 < j_2 < ... < j_k \leq n$.
  Each pair $(i_t, j_t)$ is called a \emph{step} of $\L$.
\end{definition}

\begin{definition}
  The \emph{ladder matrices} on $\L$ is the subspace
  \[
    M_\L = \Span \bigcup_{t = 1}^k
    \left\{
      e_{i,j} \middle\vert
      1 \leq i \leq i_t \text{ and } j_t \leq j \leq 
    \right\}
    \text{.}
  \]
\end{definition}

\begin{definition}
  A ladder $\L$ is called \emph{upper triangular} if $i_t < j_{t+1}$
  for $t = 1, 2, ..., k-1$.
\end{definition}

\begin{theorem}[\cite{brice2015zero}]\label{thm:A}
  Let $\L$ be a ladder on $n$.
  $M_\L$ is closed under matrix multiplication
  (and subsequently under the Lie bracket)
  if and only if $\L$ is upper triangular.
\end{theorem}

We remind the reader that if $x, y \in M^{n \times n}_F$, then the Lie
bracket of $x$ and $y$, denoted $[x, y]$, is the matrix $xy - yx$.
A subspace of $M^{n \times n}_\L$ closed under $[\cdot, \cdot]$ is
termed a \emph{Lie algebra}.

In light of Theorem \ref{thm:A},
whenever $\L$ is upper triangular we will call $M_\L$ the
\emph{matrix algebra on $\L$} in case we are considering $M_\L$ as an
algebra under matrix multiplication or the
\emph{matrix Lie algebra on $\L$} in case we are considering $M_\L$ as
an algebra under the Lie bracket.

The following proposition establishes that the class of block upper
triangular matrix algebras is a subclass of the class of ladder matrix
algebras.

\begin{proposition}\label{prop:A}
  Let $\q \subseteq M^{n \times n}_F$ be a block upper triangular matrix
  algebra (res. Lie algebra).
  There is an upper triangular ladder $\L$ such that $\q = M_\L$.
\end{proposition}

\begin{proof}
  Block upper triangular matrix algebras (res. Lie algebras)
  correspond with partitions of $n$ \cite{knapp2002lie}.
  Let $\pi = (n_1, n_2, ..., n_k)$ be the partition of $n$ corresponding
  to $\q$.
  Let
  \[
    \L = \left\{
    \left( \sum_{i = 1}^t n_i, 1 + \sum_{i = 1}^{t - 1} n_i \right)
    \middle\vert 1 \leq t \leq k
    \right\}
  \]
  where $\sum_{i=1}^0 n_i$ should be understood to be $0$.
  $\L$ is upper triangular by construction, and furthermore is
  constructed so that $\q = M_\L$.
\end{proof}

Stated perhaps more clearly, the block upper triangular matrix algebras
are precisely the ladder matrix algebras where $j_{t + 1} = i_t + 1$ for
$t = 1, 2, ..., k - 1$.

\begin{definition}
  An \emph{algebra} over $F$ is a pair $(A, \mu)$ where $A$ is a vector
  space over $F$ and $\mu : A \otimes A \to A$ is an $F$-linear map.
  The image of $\mu$ is denoted by $A^2$.
\end{definition}

This definition of algebra does not assume that the multiplication map
$\mu$ is associative.
This definition is chosen because it is agnostic to whether we are
considering $M_\L$ as an associative algebra under
$\mu : x \otimes y \mapsto xy$
or as a Lie algebra under
$\mu : x \otimes y \mapsto [x,y]$.

\begin{definition}
  An algebra is called \emph{zero product determined} if for each
  $F$-linear map $\varphi : A \otimes A \to X$
  (where $X$ is an arbitrary vector space over $F$)
  the condition
  \begin{equation}\label{eq:A}
    \forall x, y \in A,
    \varphi( x \otimes y) = 0 \text{ whenever } \mu( x \otimes y) = 0
  \end{equation}
  ensures that $\varphi$ factors through $\mu$.
\end{definition}

\[
  \xymatrix{
    A \otimes A \ar[d]_\mu \ar[dr]^\varphi & \\
    A^2 \ar@{-->}[r]_f & X
  }
\]

A linear map satisfying condition \ref{eq:A} is said to \emph{preserve
zero products}.
By $\varphi$ \emph{factors through} $\mu$ it is meant that there is a
linear map $f : A^2 \to X$ such that $\varphi = f \circ \mu$,
as illustrated above.
If $\varphi$ factors through $\mu$, then condition \ref{eq:A} holds
trivially.
We note that in case $(A, \mu)$ is zero product determined
and $\varphi : A \otimes A \to X$ preserves zero products,
then the map $f$ such that
$\varphi = f \circ \mu$
is uniquely determined.

The notion of a zero-product determined algebra was introduced by
Matej Bre{\v{s}}ar, Mateja Gra{\v{s}}i{\v{c}}, and Juana S{\'a}nchez
Ortega in \cite{brevsar2009zero} to further the study of
near-homomorphisms on Banach algebras.
We present below the results of interest to us in this paper.

\begin{theorem}[\cite{brevsar2009zero}]\label{thm:B}
  $M^{n \times n}_F$ considered as an algebra under either matrix
  multiplication or the Lie bracket is zero product determined.
\end{theorem}

\begin{theorem}[\cite{gravsivc2010zero}]
  The classical Lie algebras are zero product determined.
\end{theorem}

\begin{theorem}[\cite{wang2011class}]
  The simple Lie algebras over $\mathbb{C}$ and their parabolic
  subalgebras are zero product determined.
\end{theorem}

\begin{theorem}[\cite{brice2015zero}]\label{thm:D}
  An abelien Lie algebra is zero product determined.
\end{theorem}

\begin{theorem}[\cite{brice2015zero}]
  If $\L$ is upper triangular, then $M_\L$ under matrix multiplication
  is zero product determined.
\end{theorem}

Recall that $A \otimes A = \Span \{ x \otimes y \vert x, y \in A\}$.
Members of $A \otimes A$ of the form $x \otimes y$ with $x, y \in A$
are called \emph{rank-one} tensors.
We will make extensive use of the following theorem.

\begin{theorem}[\cite{brice2015zero}]\label{thm:C}
  An algebra
  $(A, \mu)$ is zero product determined if and only if $\Ker \mu$ is
  generated by rank-one tensors.
\end{theorem}

We note that while $A \otimes A$ is generated by rank-one tensors by
definition, an arbitrary subspace of $A \otimes A$ need not be
generated by the rank-one tensors it contains.

\section{Main Result}

We state and prove our main result.

\begin{proposition}
  Let $\L$ be a $1$-step ladder on $n$.
  The ladder matrix Lie algebra $M_\L$ is zero product determined.
\end{proposition}

\begin{proof}
  Let $\L = \{ (i_1, j_1) \}$.
  If $i_1 < j_1$, then $M_\L$ is abelien and is zero product determined
  by Theorem \ref{thm:D}.
  We assume without loss of generality that $i_1 \geq j_1$.

  Let $\mu : \sum_i x_i \otimes y_i \mapsto \sum_i [x_i, y_i]$.
  In light of Theorem \ref{thm:C}, our task is to construct a basis of
  $\Ker \mu$ consisting of elements of $M_\L \otimes M_\L$ of the form
  $x \otimes y$ with $x, y \in M_\L$.

  We partition $M_\L$ into blocks of size $n_1 = j_1 - 1 \geq 0$,
  $n_2 = i_1 - j_1 + 1 > 0$, and $n_3 = n - i_1 \geq 0$ so that
  $n_1 + n_2 + n_3 = n$.
  Under this block scheme, $M_\L$ has the form
  \[
    M_\L =
    \bordermatrix{
          & n_1 & n_2 & n_3 \cr
      n_1 & 0   & \l  & \a  \cr
      n_2 & 0   & \h  & \r  \cr
      n_3 & 0   & 0   & 0
    }
    \text{,}
  \]
  or in case $n_1 = 0$
  \[
    M_\L =
    \bordermatrix{
          & n_2 & n_3 \cr
      n_2 & \h  & \r  \cr
      n_3 & 0   & 0
    }
    \text{,}
  \]
  or in case $n_3 = 0$
  \[
    M_\L =
    \bordermatrix{
          & n_1 & n_2 \cr
      n_1 & 0   & \l  \cr
      n_2 & 0   & \h
    }
    \text{,}
  \]
  where each of $\h$, $\l$, $\r$, and $\a$ is a subalgebra consisting of
  the full matrix subspace of the appropriate size.
  All three cases are treated simultaneously by the below argument.

  $M_\L$ admits the structural decomposition
  \[
    M_\L = \h \ltimes \big( (\l \dot{+} \r) \ltimes \a \big)
  \]
  obeying multiplication containment relations below.
  \[
    \begin{array}{c|cccc}
      [\cdot,\cdot] & \h & \l & \r & \a \\
      \hline
      \h & \h & \l & \r & 0 \\
      \l & \l & 0 & \a & 0 \\
      \r & \r & \a & 0 & 0 \\
      \a & 0 & 0 & 0 & 0
    \end{array}
  \]
  (where $\l = \a = 0$ in case $n_1 = 0$
  and $\r = \a = 0$ in case $n_3 = 0$.)

  We require the dimension of $\Ker \mu$.

  We see that for $h \in \h$ and $r \in \r$ we have $[h, r] = hr$, since
  $rh = 0$, and similarly with $l \in \l$ we have $[h, l] = -lh$.
  Thus $[\h, \r] = \r$ and $[\h, \l] = \l$.
  Furthermore, for $l \in \l$ and $r \in \r$, we have $[l, r] = lr$,
  whereby $[\l, \r] = \a$.
  Finally, $[\h, \h]$ produces only the traceless matrices,
  thus $\dim [\h, \h] = \dim \h - 1$.

  In light of these observations, we find that $\Ker \mu$ has dimension
  \begin{gather*}
    n_1^2n_2^2 + 2n_1^2n_2n_3 + n_1^2n_3^2 + 2n_1n_2^3 + 4n_1n_2^2n_3
    + 2n_1n_2n_3^2 \\
    - n_1n_2 - n_1n_3 + n_2^4 + 2n_2^3n_3 + n_2^2n_3^2 - n_2^2 - n_2n_3
    + 1\text{.}
  \end{gather*}

  Each pairing of subspaces that is killed by the bracket yields its
  full basis of rank-one tensors to $\Ker \mu$. We have:
  \[
    \begin{array}{c|c}
      \text{Subspace pair} & \text{Rank-one tensors contributed} \\
      \hline
      \mu (\h \otimes \a) = 0 = \mu (\a \otimes \h) & 2n_1n_2^2n_3 \\
      \mu (\l \otimes \a) = 0 = \mu (\a \otimes \l) & 2n_1^2n_2n_3 \\
      \mu (\r \otimes \a) = 0 = \mu (\a \otimes \r) & 2n_1n_2n_3^2 \\
      \mu (\a \otimes \a) = 0 & n_1^2n_3^2 \\
      \mu (\l \otimes \l) = 0 & n_1^2n_2^2 \\
      \mu (\r \otimes \r) = 0 & n_2^2n_3^2
    \end{array}
  \]

  Further, $\h$ is isomorphic to $M^{n_2 \times n_2}_F$, which is zero
  product determined as a Lie algebra by Theorem \ref{thm:B}. By Theorem
  \ref{thm:C} there are $n_2^4 - n_2^2 + 1$ rank-one tensors in
  $\h \otimes \h$ that $\mu$ kills.
  The above listed rank-one tensors in $\Ker \mu$ are linearly
  independent by construction from block pairings.
  This leaves
  \[
    2 n_1 n_2^3 + 2 n_2^3 n_3 + 2 n_1 n_2^2 n_3
    - n_1 n_2 - n_1 n_3 - n_2 n_3
  \]
  rank-one tensors in $\Ker \mu$ we have left to construct.

  We examine $\h \otimes \r$, $\r \otimes \h$,
  and $(\h \dot{+} \r) \otimes (\h \dot{+} \r)$.
  We will find that these subspaces contribute $2n_2^3n_3 - n_2n_3$
  tensors to our basis.

  Consider the $2n_2^3n_3 - 2n_2^2n_3$ tensors
  \[
    T_{i,j,l,q} = e_{i, j} \otimes
    e_{l, q} \in \h \otimes \r
  \]
  and
  \[
    T^{i,j,l,q} = e_{l, q} \otimes
    e_{i, j} \in \r \otimes \h
  \]
  for $i,j,l \in (n_1, n_1 + n_2]$
  and $q \in (n_1 + n_2, n_1 + n_2 + n_3]$
  with $j \neq l$.

  Additionally, we have $2n_2^2n_3 - 2n_2n_3$ tensors
  \[
    S_{i,j,q} =
    \left( e_{i, j} - e_{i, j+1} \right)
    \otimes
    \left( e_{j, q} + e_{j + 1, q} \right)
    \in \h \otimes \r
  \]
  and
  \[
    S^{i,j,q} =
    \left( e_{j, q} + e_{j + 1, q} \right)
    \otimes
    \left( e_{i, j} - e_{i, j+1} \right)
    \in \r \otimes \h
  \]
  with $i \in (n_1, n_1 + n_2]$,
  $j \in (n_1, n_1 + n_2 - 1]$,
  and $q \in (n_1 + n_2, n_1 + n_2 + n_3]$.

  Finally, we have $n_2n_3$ tensors of the form
  \[
    R(i,q) =
    \left( e_{i, i} + e_{i, q} \right)
    \otimes
    \left( e_{i, i} + e_{i, q} \right)
    \in (\h \dot{+} \r) \otimes (\h \dot{+} \r)
  \]
  for $i \in (n_1, n_1 + n_2]$
  and $q \in (n_1 + n_2, n_1 + n_2 + n_3]$,
  giving the desired $2n_2^3n_3 - n_2n_3$ rank-one tensors.
  By applying $\mu (x \otimes y) = [x,y]$,
  we see that each tensor above is in $\Ker \mu$.
  We must show that these tensors are linearly independent.

  Expanding $S_{i,j,q}$ we see that
  \[
    S_{i,j,q} =
    \underbrace{
      e_{i, j} \otimes e_{j, q} -
      e_{i, j + 1} \otimes e_{j + 1, q}
    }_{ \notin \Span \{ T_{i,j,l,q} \} }
    +
    \underbrace{
      e_{i, j} \otimes e_{j + 1, q} -
      e_{i, j + 1} \otimes e_{j, q}
    }_{ \in \Span \{ T_{i,j,l,q} \} }
  \]
  is not in the span of the $T_{i,j,l,q}$ tensors.
  A similar observation shows that $S^{i,j,q}$ is not in the span of the
  $T^{i,j,l,q}$ tensors.

  Expanding $R(i,q)$ we have
  \[
    R(i,q)
    = \underbrace{ e_{i,i} \otimes e_{i,i} }_{ \in \h \otimes \h }
    + \underbrace{ e_{i,q} \otimes e_{i,q} }_{ \in \r \otimes \r }
    + \underbrace{
      e_{i,i} \otimes e_{i,q} + e_{i,q} \otimes e_{i,i}
    }_{ \in \h \otimes \r \dot{+} \r \otimes \h }\text{.}
  \]
  Since $e_{i,i} \otimes e_{i,i}$ and $e_{i,q} \otimes e_{i,q}$ are in
  $\h \otimes \h$ and $\r \otimes \r$, respectively,
  and since tensors from those blocks have been accounted for above,
  we may subtract those terms, leaving
  $R'(i,q) = e_{i,i} \otimes e_{i,q} + e_{i,q} \otimes e_{i,i}$.
  $R'(i,q)$ is not in the span of
  $\left\{ T_{i,j,l,q}, T^{i,j,l,q} \right\}$
  since we require $j \neq l$ in $T_{i,j,l,q}$ and $T^{i,j,l,q}$.

  Now, consider $S_{i,i,q} + S^{i,i,q}$ where $i < n_1 + n_2$
  ($R(i,j)$ is linearly independent of the $S_{i,j,q}$ and $S^{i,j,q}$
  tensors in case $i = n_1 + n_2$, since we require
  $j \leq n_1 + n_2 - 1$ in $S_{i,j,q}$ and $S^{i,j,q}$).
  We have
  \[
    S_{i,i,q} + S^{i,i,q}
    = \underbrace{
      e_{i,i} \otimes e_{i,q} + e_{i,q} \otimes e_{i,i}
    }_{ = R'(i,q) }
    + T
    - \left(
      e_{i,i+1} \otimes e_{i+1,q}
      + e_{i+1,q} \otimes e_{i, i+1}
    \right)
  \]
  with $T \in \Span \left\{ T_{i,j,l,q}, T^{i,j,l,q} \right\}$,
  so we have
  \[
    R'(i,q)
    = S_{i,i,q} + S^{i,i,q} - T
    + e_{i,i+1} \otimes e_{i+1,q} + e_{i+1,q} \otimes e_{i, i+1}\text{.}
  \]

  Write
  $R''(i,q)
  = e_{i,i+1} \otimes e_{i+1,q}
  + e_{i+1,q} \otimes e_{i, i+1}$.
  Now, if $i = n_1 + n_2-1$ we are done (as above).
  If $i < n_1 + n_2-1$ we may reduce $R''(i,q)$ using the same method
  just employed, and so by induction we are done.
  That is to say that $T_{i,j,l,q}$, $T^{i,j,l,q}$,
  $S_{i,j,q}$, $S^{i,j,q}$, and $R(i,j)$ are linearly independent.

  Next, we examine $\h \otimes \l$, $\l \otimes \h$,
  and $(\h \dot{+} \l) \otimes (\h \dot{+} \l)$.
  The consideration of these subspaces is symmetric with the subspaces
  considered above,
  and so we will find that these subspaces contribute
  $2n_1n_2^3 - n_1n_2$ tensors to our basis of $\Ker \mu$.

  Finally, we examine $\l \otimes \r$, $\r \otimes \l$,
  and $(\l \dot{+} \r) \otimes (\l \dot{+} \r)$.
  We proceed similarly to the discussion of $\h$ and $\r$ above,
  and we will find that $\l$ and $\r$ contribute the remaining
  $2n_1n_2^2n_3 - n_1n_3$ rank-one tensors needed to span $\Ker \mu$.

  Consider the $2n_1n_2^2n_3 - 2n_1n_2n_3$ tensors
  \[
    U_{i,j,l,q} = e_{i, j} \otimes
    e_{l, q} \in \l \otimes \r
  \]
  and
  \[
    U^{i,j,l,q} = e_{l, q} \otimes
    e_{i, j} \in \r \otimes \l
  \]
  for $i \in (0, n_1]$,
  $j, l \in (n_1, n_1 + n_2]$,
  and $q \in (n_1 + n_2, n_1 + n_2 + n_3]$
  with $j \neq l$.

  Additionally, we have $2n_1n_2n_3 - 2n_1n_3$ tensors
  \[
    V_{i,j,q} =
    \left( e_{i, j} - e_{i, j+1} \right)
    \otimes
    \left( e_{j, q} + e_{j + 1, q} \right)
    \in \l \otimes \r
  \]
  and
  \[
    V^{i,j,q} =
    \left( e_{j, q} + e_{j + 1, q} \right)
    \otimes
    \left( e_{i, j} - e_{i, j+1} \right)
    \in \r \otimes \l
  \]
  with $i \in (0, n_1]$,
  $j \in (n_1, n_1 + n_2 - 1]$,
  and $q \in (n_1 + n_2, n_1 + n_2 + n_3]$.

  Finally, we have $n_1n_3$ tensors of the form
  \[
    W(i,q) =
    \left( e_{i, n_1 + n_2} + e_{n_1 + n_2, q} \right)
    \otimes
    \left( e_{i, n_1 + n_2} + e_{n_1 + n_2, q} \right)
    \in (\l \dot{+} \r) \otimes (\l \dot{+} \r)
  \]
  for $i \in (0, n_1]$
  and $q \in (n_1 + n_2, n_1 + n_2 + n_3]$,
  giving the remaining $2n_1n_2^2n_3 - n_1n_3$ rank-one tensors.
  Again, the above tensors were chosen so that applying
  $\mu (x \otimes y) = [x,y]$ results in $0$.
  Below we verify that they are linearly independent.

  Expanding $V_{i,j,q}$ we see that
  \[
    V_{i,j,q} =
    \underbrace{
      e_{i, j} \otimes e_{j, q} -
      e_{i, j + 1} \otimes e_{j + 1, q}
    }_{ \notin \Span \{ U_{i,j,l,q} \} }
    +
    \underbrace{
      e_{i, j} \otimes e_{j + 1, q} -
      e_{i, j + 1} \otimes e_{j, q}
    }_{ \in \Span \{ U_{i,j,l,q} \} }
  \]
  is not in the span of the $U_{i,j,l,q}$ tensors.
  A similar observation shows that $V^{i,j,q}$ is not in the span of the
  $U^{i,j,l,q}$ tensors.

  Expanding $W(i,q)$ we have
  \begin{align*}
    W(i,q)
    &= \underbrace{
      e_{i,n_1 + n_2} \otimes e_{i,n_1 + n_2}
    }_{ \in \l \otimes \l }
    + \underbrace{
      e_{n_1 + n_2,q} \otimes e_{n_1 + n_2,q}
    }_{ \in \r \otimes \r }
    \\ &+ \underbrace{
      e_{i,n_1 + n_2} \otimes e_{n_1 + n_2,q}
      + e_{n_1 + n_2,q} \otimes e_{i,n_1 + n_2}
    }_{ \in \l \otimes \r \dot{+} \r \otimes \l }\text{.}
  \end{align*}
  $\l \otimes \l$ and $\r \otimes \r$ are accounted for above,
  so we may subtract their terms, leaving
  \[
    W'(i,q) =
    e_{i,n_1 + n_2} \otimes e_{n_1 + n_2,q}
    + e_{n_1 + n_2,q} \otimes e_{i,n_1 + n_2}
    \text{.}
  \]
  $W'(i,q)$ is not in the span of
  $\left\{ U_{i,j,l,q}, U^{i,j,l,q} \right\}$
  since we require $j \neq l$ in $U_{i,j,l,q}$ and $U^{i,j,l,q}$.
  We also see immediately that $W'(i,q)$ is not in the span of
  $\left\{ V_{i,j,q}, V^{i,j,q} \right\}$
  since we require $j < n_1 + n_2$ in $V_{i,j,q}$ and $V^{i,j,q}$.
  Thus we have that $U_{i,j,l,q}$, $U^{i,j,l,q}$,
  $V_{i,j,q}$, $V^{i,j,q}$, and $W(i,j)$ are linearly independent.

  Having explicitly constructed a basis for $\Ker \mu$ consisting of
  rank-one tensors, the proof is complete.
\end{proof}

\section*{Acknowledgment}

The author wishes to express his gratitude to the referee for their
helpful comments, corrections, and suggestions.

\printbibliography

\end{document}